\documentclass[11pt]{amsart}
\usepackage{amsfonts,amsmath,amssymb,epsfig,cite,graphicx,hyperref,
color, esint, fancyhdr, enumerate, latexsym, amsrefs, mathrsfs}

\newtheorem{thm}{Theorem}[section]
\newtheorem{prop}[thm]{Proposition}
\newtheorem{lem}[thm]{Lemma}
\newtheorem{cor}[thm]{Corollary}

\theoremstyle{definition}

\newcommand{\mbb}{\mathbb}

\newcommand{\pa}{\partial}
\newcommand{\mf}{\mathbb}

\newcommand{\Om}{\Omega}
\newcommand{\al}{\alpha}
\newcommand{\be}{\beta}

\newcommand{\ti}{\tilde}

\renewcommand{\Re}{\operatorname{Re}}

\newcommand{\diam}{\operatorname{diam}}
\newcommand{\dist}{\operatorname{dist}}

\numberwithin{equation}{section}
\makeatletter
\@namedef{subjclassname@2020}{%
  \textup{2020} Mathematics Subject Classification}
\makeatother

\hypersetup{
    colorlinks,
    citecolor=blue,
    filecolor=black, 
    linkcolor=red,
    urlcolor=black
}

\title[Geometric estimates for Eisenman volume elements]{Geometric estimates and comparability of Eisenman volume elements with the Bergman kernel on ($\mathbb{C}$-)convex domains}

\author{Debaprasanna Kar}
\address{Department of Mathematics, Indian Institute of Science, Bangalore-560012, India}
\email{debaprasanna@iisc.ac.in}

\begin{document}
\maketitle

\begin{abstract}
We establish geometric upper and lower estimates for the Carath\'eodory and Kobayashi-Eisenman volume elements on the class of non-degenerate convex domains, as well as on the more general class of non-degenerate $\mathbb{C}$-convex domains. As a consequence, we obtain explicit universal lower bounds for the quotient invariant both on non-degenerate convex and $\mathbb{C}$-convex domains. Here the bounds we derive, for the above mentioned classes in $\mathbb{C}^{n}$, only depend on the dimension $n$ for a fixed $n\geq 2$. Finally, it is shown that the Bergman kernel is comparable with these volume elements up to small/large constants depending only on $n$.
\end{abstract}

\section{Introduction, Motivation and Main results} \label{intro}

We denote the unit disc in the complex plane by $\mbb D$, the unit ball and the unit polydisc in $\mbb C^n$ by $\mbb B^n$ and $\mbb D^n$ respectively. For a domain $D\subset\mbb{C}^n$, the \textit{Carath\'eodory} and \textit{Kobayashi-Eisenman volume elements} on $D$ at a point $z^0 \in D$ are defined respectively by
\begin{align*}
c_D\left(z^0\right) & = \sup \Big\{ \big\vert \det \psi^{\prime}\left(z^0\right) \big\vert^2 : \psi \in \mathcal{O}(D,\mbb{B}^n), \psi\left(z^0\right) = 0 \Big\},\\
k_D\left(z^0\right) &= \inf \Big\{\big\vert \det \psi^{\prime}(0)\big\vert^{-2 } : \psi \in \mathcal{O}(\mbb{B}^n, D), \psi(0)=z^0 \Big\}.
\end{align*}
In the definitions above, $\det \psi'$ denotes the complex Jacobian determinant of a holomorphic map $\psi$. By a normal family argument one can show that $c_D\left(z^0\right)$ is always attained, and if $D$ is taut, then $k_D\left(z^0\right)$ is also attained. It is possible that $c_D$ can identically be 0, for example if $D=\mbb C$; but is strictly positive if $D$ is not a Liouville domain. Likewise $k_D$ can also vanish, but if $D$ is bounded, by invoking Cauchy's estimates we see that $k_D>0$. Under a holomorphic map $F:D_1 \to D_2$, the volume elements satisfy the rule
\begin{align}\label{ineq}
v_{D_1}(z) \geq \big\vert \det F^{\prime}(z)\big\vert^2 v_{D_2}\big(F(z)\big),
\end{align}
where $v=c$ or $k$. This, in particular, implies the monotonicity property of  the volume elements, i.e. for domains $D_1\subset D_2 \subset \mbb C^n,\, v_{D_1}(z)\geq v_{D_2}(z)$ for all $z\in D_1$. Again, the inequality in (\ref{ineq}) becomes an equality when $F$ is a biholomorphism. Therefore, if $k_D$ is non-vanishing (which is the case when $D$ is bounded or a taut domain), then
\[
q_D(z)=\frac{c_D(z)}{k_D(z)}
\]
is a biholomorphic invariant and is called the \textit{quotient invariant}. The Carath\'eodory-Cartan-Kaup-Wu theorem shows that the Kobayashi-Eisenman volume element always dominates the Carath\'eodory one, and therefore $q_D \leq 1$ for any domain $D\subset \mbb C^n$. If $D=\mbb{B}^n$, then
\[
c_{\mbb{B}^n}(z)=k_{\mbb{B}^n}(z)=\big(1-\vert z \vert^2\big)^{-n-1},
\]
and thus $q_{\mbb{B}^n}$ is identically equal to $1$. It is a remarkable result that if $D$ is any domain in $\mbb{C}^n$ and $q_D\left(z^0\right)=1$ for some point $z^0 \in D$, then $q_D \equiv 1$ and $D$ is biholomorphic to $\mbb{B}^n$. This result has been proved by many authors with several conditions assumed on the domain $D$ (cf. \cite{Wong} when $D$ is bounded and complete hyperbolic, \cite{Ro} when $D$ is only assumed to be bounded, \cite{Dek} when $D$ is hyperbolic). Later the work of Graham and Wu \cite{Gr-Wu} showed that no assumption on $D$ is required for the result to be true, in fact, it holds true for complex manifolds. Thus $q_D$ measures the extent to which the Riemann mapping theorem fails for $D$.

Recall that a domain $D\subset \mbb C^n$ is called \textit{$\mbb C$-convex} if $D\cap l$ is a simply connected domain for each complex affine line $l$ passing through $D$. Although $\mbb C$-convexity can be thought as a complex analog of the concept of convexity, the inherent geometry and regularity phenomena occurring in case of $\mbb C$-convex domains are much more intriguing than those for the convex ones. One of the crucial properties of $\mbb C$-convexity that we will use later is that $\mbb C$-convex domains are necessarily \textit{linearly convex}. A domain $D\subset \mbb C^n$ is said to be linearly convex if for any $p\in \mbb C^n\setminus D$ there exists a complex hyperplane passing through $p$ which does not intersect $D$. We have the following set of inclusions
\[
\text{convex domains}\subset\mbb C\text{-convex domains} \subset \text{linearly convex domains}.
\]
Moreover, the above inclusions are strict. For example, any planar domain vacuously satisfies the criteria of a linearly convex domain, whereas any planar annulus is not $\mbb C$-convex. Similarly, any contractible non-convex planar domain is $\mbb C$-convex, but obviously not convex by consideration. In \cite{Nik-Pflug-Zwonek}, Nikolov \textit{et al.} showed that the symmetrized bidisc is a $\mbb C$-convex domain in $\mbb C^2$ which is not biholomorphic to any convex domain.

We say that a domain is \textit{non-degenerate} if it contains no complex lines. The purpose of this article is to establish geometric upper and lower estimates for the Carath\'eodory and Kobayashi-Eisenman volume elements on non-degenerate convex and $\mbb C$-convex domains. This set of estimates, as we will see, captures the asymptotic behavior of the volume elements on those aforementioned domains. In this process, we will derive two separate uniform lower bounds, one for the class of non-degenerate convex domains and the other for non-degenerate $\mbb C$-convex domains, for the quotient invariant. Here the bounds are uniform in the sense that  the derived lower bounds hold true for all domains in $\mbb C^n$ in the stated class for a fixed $n\geq 2$. More specifically, given $n\geq 2$, we provide two positive real numbers $\mu_n$ and $\nu_n$ such that for any non-degenerate convex domain $G\subset \mbb C^n$ and for any non-degenerate $\mbb C$-convex domain $\ti G\subset \mbb C^n$,
\[
q_G(z)\geq \mu_n,\; q_{\ti G}(w)\geq \nu_n \quad \forall z\in G\;\,\text{and}\;\,\forall w\in \ti G.
\]
Note that the quotient invariant, as well as the volume elements, have been extensively studied near the boundary of several classes of domains. In \cites{Cheung-Wong, Gr-Kr1, Ma} the authors studied the boundary behavior on strongly pseudoconvex domains. Nikolov, in \cite{Nik-sq}, computed non-tangential boundary asymptotics of both the volume elements near h-extendible boundary points. Recently in \cite{Borah-Kar-1}, boundary asymptotics of the Kobayashi-Eisenman volume element have been obtained in terms of distinguished polydiscs of McNeal and Catlin near convex finite type and Levi corank one boundary points respectively.

The results in this note are motivated by a recent article of Bharali and Nikolov \cite{Bharali-Nik}, where the authors prove \textit{homogeneous regularity} of non-degenerate $\mbb C$-convex domains by providing explicit lower bounds for the squeezing function. A domain $D\subset \mbb C^n$ is said to be homogeneous regular if the squeezing function $s_D$ satisfies $s_D(z)\geq c$ for a fixed positive constant $c$, for all $z\in D$. We preferably want correlations between the squeezing function and the volume elements since these objects are defined through maps from/into the unit ball, but unfortunately due to lack of a higher-dimensional analogue of the Koebe 1/4 theorem, we do not have much! However, in 2013, Deng, Guan and Zhang established a relationship between the squeezing function and the volume elements on bounded domains in $\mbb C^n$ (see Theorem~3.1 of \cite{DGZ}). In particular, on a bounded domain $D\subset \mbb C^n$, they showed
\begin{align*}
s_D^{2n}(z) v_D'(z)\leq v_D(z) \leq \dfrac{1}{s_D^{2n}(z)} v_D'(z),
\end{align*}
for $v_D, v_D'\in \{c_D, k_D\}$. The inequality above implicitly infer that if a bounded domain $D$ is homogeneous regular, then the quotient invariant on $D$ is bounded below by a positive constant. The results we obtain in this article will extend this phenomena to certain classes of unbounded domains. To be precise we will show that on non-degenerate convex and $\mbb C$-convex domains, which may not be bounded and which are shown to be homogeneous regular in \cite{Bharali-Nik}, the quotient invariant is bounded below by uniform positive constants (see Corollaries~\ref{q convex} and~\ref{q C-convex} below).

\medskip
Let $D\subset \mbb C^n$ be a non-degenerate domain and $z^0\in D$. We denote the boundary of $D$ as $\pa D$ and $H_0:=\mbb C^n$. Let $\tau_{1,D}(z^0)$ be the minimum distance of $z^0$ to $\pa D$ and $p^1$ be a point on $\pa D$ realising this distance. Let $H_{1}$ be the complex hyperplane passing through $z^0$ orthogonal to the vector $p^1-z^0$. Clearly,
\begin{align*}
    H_0= H_1 \oplus \text{span}_{\mbb C}\left\{p^1-z^0\right\}.
\end{align*}
Now, let $\tau_{2,D}(z^0)$ be the minimum distance of $z^0$ to $\pa D \cap H_1$ and $p^2$ be a point on $\pa D \cap H_1$ realising this distance. Then we carry out similar construction, replacing $H_1$ by $H_2$ in the next step, where
\begin{align*}
    H_1=H_2 \oplus \text{span}_{\mbb C}\left\{ p^2-z^0\right\}.
\end{align*}
Continuing in this way, we define $\mbb C$-linear subspaces $H_0,\ldots,H_{n-1}$, positive numbers $\tau_{1,D}(z^0),\ldots,$ $\tau_{n,D}(z^0)$, and the points $p^1,\ldots,p^n$ on $\pa D$. Here are some observations: the choices of $p^1,\ldots,p^n$ are not unique for a fixed $z^0$; the points $\{p^1,\ldots,p^n\}\subset \pa D$ and the subspaces $H_0,\ldots,H_{n-1}$ obviously vary depending on different points $z\in D$; the numbers $\tau_{1,D}(z^0)=\|p^1-z^0\|,\ldots,\tau_{n,D}(z^0)=\|p^n-z^0\|$ are all finite because of the assumption that $D$ is non-degenerate. Define
\begin{align*}
p_D\left(z^0\right):= \tau_{1,D}\left(z^0\right)\cdots\tau_{n,D}\left(z^0\right).
\end{align*}
We have the following geometric estimates for the volume elements on the class of convex domains:

\begin{thm}\label{ge convex}
Let $D$ be a non-degenerate convex domain in $\mbb C^n, n\geq 2$ and $z\in D$. Then
\begin{align*}
\dfrac{1}{(4n)^n} \leq v_D(z) p_D^2(z) \leq \left(\dfrac{4^n-1}{3}\right)^n,
\end{align*}
for $v_D=c_D$ or $k_D$.
\end{thm}

Similar to $p_D$, there are several such geometrically defined numbers introduced (cf. \cite{Nik-Pflug-2003, Nik-Pflug-Zwonek2, Chen-1989, McNeal-2001}) to bound the Bergman kernel function in terms of small/large constants on various pseudoconvex domains. The following result is the counterpart of Theorem~\ref{ge convex} on $\mbb C$-convex domains:

\begin{thm}\label{ge C-convex}
Let $D$ be a non-degenerate $\mbb C$-convex domain in $\mbb C^n, n\geq 2$ and $z\in D$. Then
\begin{align*}
\dfrac{1}{(16n)^n} \leq v_D(z) p_D^2(z) \leq \left(\dfrac{4^n-1}{3}\right)^n,
\end{align*}
for $v_D=c_D$ or $k_D$.
\end{thm}

As consequences of Theorems~\ref{ge convex} and~\ref{ge C-convex}, the next two results derive uniform lower bounds for the quotient invariant on the classes of non-degenerate convex and non-degenerate $\mbb C$-convex domains respectively. Consider the first inequality in Theorem~\ref{ge convex} for $v_D=c_D$ and the second inequality there for $v_D=k_D$, and then divide them up to obtain:

\begin{cor}\label{q convex}
The quotient invariant $q_D$ admits the following lower bound on any non-degenerate convex domain $D \subset \mbb C^n,\, n\geq 2$:
\[
q_D(z)\geq \left(\dfrac{3}{4n(4^n-1)}\right)^n.
\]
\end{cor}

The same argument, when applied to Theorem~\ref{ge C-convex}, produces the following uniform lower bound on $\mbb C$-convex domains in $\mbb C^n$:

\begin{cor}\label{q C-convex}
The quotient invariant $q_D$ admits the following lower bound on any non-degenerate $\mbb C$-convex domain $D \subset \mbb C^n,\, n\geq 2$:
\[
q_D(z)\geq \left(\dfrac{3}{16n(4^n-1)}\right)^n.
\]
\end{cor}

Here one can easily notice that the lower bounds obtained in the setting of convex domains are much stricter compared to those in the setting of $\mbb C$-convex domains (for example: the bound for the quotient invariant obtained above on the class of non-degenerate convex domains is $4^n$-times greater than that on non-degenerate $\mbb C$-convex domains). This should not come as a surprise to anyone since any convex domain is $\mbb C$-convex.

The next two results are also consequences of Theorems~\ref{ge convex} and~\ref{ge C-convex} which produce lower bounds for the volume elements on bounded convex and $\mbb C$-convex domains. Here the bounds deduced are not uniform as they depend on the domain under consideration.

\begin{cor}\label{v convex}
Let $D$ be a bounded convex domain in $\mbb C^n, n \geq 2$. Then
\[
v_D(z)\geq \dfrac{1}{\left(4n\diam^2(D)\right)^n},
\]
for $v_D=c_D$ or $k_D$. Here $\diam(D)$ stands for the diameter of the domain $D$.
\end{cor}

\begin{cor}\label{v C-convex}
Let $D$ be a bounded $\mbb C$-convex domain in $\mbb C^n, n\geq 2$. Then
\[
v_D(z)\geq \dfrac{1}{\left(16n\diam^2(D)\right)^n},
\]
for $v_D=c_D$ or $k_D$.
\end{cor}

\begin{proof}[Proofs of Corollary~\ref{v convex} and Corollary~\ref{v C-convex}]
Observe that when $D$ is bounded, $$\tau_{j,D}(z)=\|p^j-z\|\leq \diam(D)<+\infty$$ for each $z\in D$ and $j=1,\ldots,n$. Thus $p_D^2(z)\leq \diam^{2n}(D)$. Combining this observation with the first inequalities from Theorem~\ref{ge convex} and Theorem~\ref{ge C-convex} proves Corollary~\ref{v convex} and Corollary~\ref{v C-convex} respectively.
\end{proof}

Before stating our next results, recall that for a domain $D\subset\mbb{C}^n$, the space
\[
A^2(D)=\left\{\text{$f: D \to \mf{C}$ holomorphic and $\|f\|^2_D:=\int_{D} \vert f\vert^2 \, dV< \infty$} \right\},
\]
where $dV$ is the Lebesgue measure on $\mf{C}^n$, is a closed subspace of $L^2(D)$, and hence is a Hilbert space. It is called the \textit{Bergman space} of $D$. $A^2(D)$ carries a reproducing kernel $K_D(z,w)$ called the \textit{Bergman kernel} for $D$. Let $K_D(z):=K_D(z,z)$ be its restriction to the diagonal of $D$. It is well-known (see \cite{Jarn-Pflug-2013}) that 
\begin{align*}
K_D(z)=\sup\big\{|f(z)|^2: f\in A^2(D), \|f\|_D\leq 1\big\}.
\end{align*}
When $D$ is bounded, one easily sees that $K_D$ is a positive function on $D$. Let $D\subset \mbb C^n$ be a convex (or, $\mbb C$-convex) domain. If $D$ contains a complex line, then it is linearly equivalent to the Cartesian product of $\mbb C$ and a convex (or, $\mbb C$-convex) domain in $\mbb C^{n-1}$. In both the cases $K_D\equiv 0$. In our consideration since $D$ is non-degenerate, it is biholomorphic to a bounded domain (cf. \cite{Jarn-Pflug-2013}), and hence $K_D>0$. 

Now we prove comparability of the Bergman kernel (on the diagonal) with the volume elements on non-degenerate convex and $\mbb C$-convex domains in $\mbb C^n$. Our aim is to bound the ratio of either of these two volume elements and the Bergman kernel by small/large constants depending only on $n$.

\begin{thm}\label{com convex}
Let $D$ be a non-degenerate convex domain in $\mbb C^n, n\geq 2$. Then
\begin{align*}
\dfrac{\pi^n}{(2n)! (2n)^n}\leq \dfrac{v_D(z)}{K_D(z)} \leq \left(\dfrac{4\pi (4^n-1)}{3}\right)^n,
\end{align*}
for $v_D=c_D$ or $k_D$.
\end{thm}

\begin{thm}\label{com C-convex}
Let $D$ be a non-degenerate $\mbb C$-convex domain in $\mbb C^n, n\geq 2$. Then
\begin{align*}
\dfrac{\pi^n}{(2n)! (8n)^n}\leq \dfrac{v_D(z)}{K_D(z)} \leq \left(\dfrac{16\pi (4^n-1)}{3}\right)^n,
\end{align*}
for $v_D=c_D$ or $k_D$.
\end{thm}

Nikolov and Thomas proved similar estimates for the ratio of the Carath\'eodory-Eisenman volume element and the Bergman kernel in light of the multidimensional Suita conjecture. In particular they showed that there exist constants $L_n>l_n>0$ depending only on $n$ such that when $D\subset \mbb C^n$ is a non-degenerate convex domain, the following estimates hold (see Corollary~6 of \cite{Nik-Pascal-2019} and the remarks given there at the top of p.~3):
\begin{align}\label{Suita convex}
\left(1/2\right)^{2n} l_n K_D(z)\leq c_D(z) \leq L_n K_D(z).
\end{align}
The exact set of inequalities hold for non-degenerate $\mbb C$-convex domains as well, just by replacing $(1/2)^{2n}$ with $(1/4)^{2n}$ in (\ref{Suita convex}). They proved the above estimates by first establishing a comparability relationship between the Carath\'eodory-Eisenman volume element and the volume of the Carath\'eodory indicatrix, and then using a multidimensional Suita conjecture result from \cite{Blocki-Zwonek-2015}. Through our proofs of Theorem~\ref{com convex} and~\ref{com C-convex}, to some extent, we provide alternative proofs to their results on the Carath\'eodory volume element. Additionally, by considering $v_D=k_D$ in Theorems~\ref{com convex} and \ref{com C-convex}, we demonstrate a version of their result on the Kobayashi volume element as well.

\medskip

\medskip

\noindent \textit{Acknowledgements}: The author would like to thank D. Borah for suggesting this problem. The author is also grateful to Prof. N. Nikolov for his kind inputs over emails which improved the contents of this article.

\section{Preparation}\label{prep}

In this segment we present some preparatory results which will eventually help us in deriving the required bounds and estimates in subsequent sections. Let $D\subset \mbb C^n$ be a non-degenerate domain and $z^0\in D$. Since the objects we are studying (e.g. the volume elements and the quotient invariant) are invariant under translation, without loss of generality we can assume $z^0=0$. Recall the construction of $H_0,\ldots,H_{n-1}$; $p^1,\ldots,p^n$; $\tau_{1,D},\ldots, \tau_{n,D}$ that we saw in the previous section, and now construct them considering the corresponding base point $z^0=0$. First assume $D$ is convex. For each $p^j$, we fix a real hyperplane $\mathcal{W}_{j-1}$ such that $\left(p^j+\mathcal{W}_{j-1}\right)$ is a supporting hyperplane of $D$ at $p^j$. If $D$ is $\mbb C$-convex, we have already seen in Sect.~\ref{intro} that $D$ is linearly convex. Hence there exists a complex hyperplane $W_{j-1}$ such that $\left(p^j+W_{j-1}\right)$ does not intersect $D$. Note that if $D$ is convex (it is also $\mbb C$-convex), the following complex hyperplanes
\begin{align}\label{hyperplane}
    W_j:= \mathcal{W}_{j} \cap i \mathcal{W}_{j}, \quad  j=0,\ldots, n-1
\end{align}
satisfy the property described in the previous line. 

Next we will see two $\mbb C$-linear maps $A$ and $T$, in the proposition stated below, which will transform our domain $D$ into a nice form. For a proof of this proposition one may refer to \cite{Bharali-Nik}, and we will mostly stick to the notations used therein with very little modifications. Denote by $e_j$ the standard $j$-th unit vector of $\mbb C^n$, $[A]_{std.}$ the matrix representation of a $\mbb C$-linear map $A$ with respect to the standard orthonormal basis $(e_1,\ldots,e_n)$, and $\langle \cdot,\cdot \rangle$ the standard Hermitian inner product on $\mbb C^n$.

\begin{prop}\label{linear maps}
Let $D$ be a non-degenerate $\mbb C$-convex domain in $\mbb C^n$, $n\geq 2$, and assume $0\in D$. For $j=1,\ldots,n$, let $p^j\in \pa D$ be as described above. Consider the invertible linear transformation $T$ given by
\[
T(z):=\sum\limits_{j=1}^n \dfrac{\langle z,p^j\rangle}{\|p^j\|^2}e_j.
\]
Fix the complex hyperplanes $W_j$ as described above, and if $D$ is convex, determine $W_j$ by the real hyperplanes $\mathcal{W}_j$ as given in (\ref{hyperplane}). Then, there exists $\mbb C$-linear transformation $A$ such that $[A]_{\text{std.}}:=[\al_{j,k}]$ is a lower triangular matrix each of whose diagonal entries is 1 and such that:
\begin{itemize}
    \item[(i)] If $D$ is convex, then
    \[ 
    A \circ T \left(p^j+\mathcal{W}_{j-1}\right)=\left\{(Z_1, \ldots, Z_n)\in \mbb C^n: \Re Z_j=1\right\}, \quad j=1,\ldots,n.
    \]
    \item[(ii)] If $D$ is $\mbb C$-convex (but not necessarily convex), then
    \[ 
    A \circ T \left(p^j+W_{j-1}\right)=\left\{(Z_1, \ldots, Z_n)\in \mbb C^n: Z_j=1\right\}, \quad j=1,\ldots,n.
    \]
\end{itemize}
Moreover, $|\al_{j,k}|\leq 1$ for $j=2,\ldots,n$, and $k=1,\ldots,j-1$ in both (i) and (ii). 
\end{prop} 
Here the linear transformation $T$ is basically a composition of a unitary map that sends $p^j$ to a point on the positive $\Re z_j$ axis, and an anisotropic dilation that stretches each coordinate in such a way that eventually $p^j$ is mapped to $e_j$, for each $j=1,\ldots,n$. By this construction, it is clear that the discs $\mbb D e_j\subset T(D)$ for $j=1,\ldots,n$. Therefore, it follows from \cite[Lemma 15]{Nik-Pflug-Zwonek2} that
\begin{align*}
    E_n:=\left\{(w_1,\ldots,w_n)\in \mbb C^n: |w_1|+\cdots+|w_n|<1\right\}\subset T(D).
\end{align*}
Note that the above inclusion is obvious when $D$ is a convex domain. 

We finally need the following lemma which will be crucial in obtaining the exact bounds in the results listed in Sect.~\ref{intro}. Here we fix $c_n:=\sqrt{4^n-1}/\sqrt{3}$.

\begin{lem}
Let $D\subset \mbb C^n$ be a domain as given in Proposition~\ref{linear maps}. With the notations used in Proposition~\ref{linear maps} and the domain $E_n$ as defined above, we have
\begin{align*}
    \left(1/c_n\right)\mbb B^n \subset A(E_n).
\end{align*}
\end{lem}

\begin{proof}
A proof of this result can be found in \cite{Bharali-Nik}. However, we present a sketch of the proof here to enhance readability. Staying true to the coordinates used in Proposition~\ref{linear maps}, we denote the coordinates of a point in $A\circ T(D)$ by $Z:=(Z_1,\ldots,Z_n)$. Write 
\begin{align*}
(w_1,\ldots,w_n):= A^{-1} (Z),
\end{align*}
which will be our coordinates for points in $T(D)$. Since $A=[\al_{j,k}]$ is a lower triangular matrix with diagonal entries 1, $A^{-1}:=[\be_{j,k}]$ will also turn out to be a lower triangular matrix with each diagonal entry 1 and for $j>k$,
\begin{itemize}
\item[($\ast$)] $\be_{j,k}$ will be a sum of $2^{j-k-1}$ monomials those are products of the terms in the set $\cup_{r>s} \{\al_{r,s}\}$.
\end{itemize}
As $|\al_{r,s}|\leq 1$ (given in Proposition~\ref{linear maps}), and since $(w_1,\ldots,w_n):= A^{-1} (Z_1,\ldots,Z_n)$, the previous remark ($\ast$) implies
\begin{align*}
& w_1=Z_1, \quad \text{and}\\
|& w_j|\leq 2^{j-2}\left|Z_1\right|+ 2^{j-3} |Z_2|+\cdots+ 2^0 |Z_{j-1}|+|Z_j|\quad \text{for}\,\,\, j=2,\ldots,n.
\end{align*}
Therefore,
\begin{align}\label{eq 20}
|w_1|+\cdots+|w_n|& \leq |Z_1|+ \sum\limits_{j=2}^n \left(|Z_j|+\sum_{k=1}^{j-1} 2^{j-k-1}|Z_k|\right)\\ \nonumber
&=\sum_{j=1}^n 2^{n-j}|Z_j|.
\end{align}
One obtains using the Cauchy--Schwarz inequality that
\begin{align*}
2^{n-1}|Z_1|+2^{n-2}|Z_2|+\cdots+|Z_n| \leq \sqrt{4^{n-1}+4^{n-2}+\cdots+1} \|Z\|,
\end{align*}
and hence the inequality in (\ref{eq 20}) implies
\begin{align}\label{eq 21}
|w_1|+\cdots+|w_n|& \leq \sqrt{4^{n-1}+4^{n-2}+\cdots+1} \|Z\|= c_n \|Z\|.
\end{align}
Note that when $c_n\|Z\|<1$, according to (\ref{eq 21}), we have $(w_1,\ldots,w_n)\in E_n$. In other words, 
\begin{align*}
A^{-1}\big((1/c_n)\mbb B^n\big)\subset E_n,
\end{align*}
which proves our claim.
\end{proof}

\section{Bounds on convex domains}

In this section we prove Theorem~\ref{ge convex} and Theorem~\ref{com convex}.

\begin{proof}[Proof of Theorem~\ref{ge convex}]
Since both the volume elements are invariant under translation, without loss of generality assume $D$ contains the origin and that $z=0$. As both $A$ and $T$ described in Sect.~\ref{prep} are biholomorphisms, we have
\begin{align}\label{eq 16}
v_D(z)=\left|\det (A\circ T)'(0)\right|^2 v_{A\circ T (D)}(0),
\end{align}
for $v=c$ or $k$. Note that, since $\det A=1$, $\det (A\circ T)'(0)=\det T'(0)$. Denoting the points $p^j:=\left(p^j_1,\ldots,p^j_n\right)$ for $j=1,\ldots,n$, one easily sees
\begin{equation}\label{eq 15}
T'(0)= 
\begin{pmatrix}
\dfrac{p_1^1}{\|p^1\|^2} & \cdots  & \dfrac{p^1_n}{\|p^1\|^2}\\
\vdots & \ddots & \vdots\\
\dfrac{p_1^n}{\|p^n\|^2} & \cdots  & \dfrac{p^n_n}{\|p^n\|^2}
\end{pmatrix}.
\end{equation}
Since $\{p^1,\ldots,p^n\}$ are orthogonal vectors, we have
\begin{align*}
\det
\begin{pmatrix}
p_1^1 & \cdots  & p^1_n\\
\vdots & \ddots & \vdots\\
p_1^n & \cdots  & p^n_n
\end{pmatrix}
=\|p^1\|\cdots \|p^n\|.
\end{align*}
Using this value in (\ref{eq 15}), one obtains 
\begin{align}\label{eq 19}
\det T'(0)=\dfrac{1}{\|p^1\|\cdots \|p^n\|},
\end{align}
and consequently Eq.~(\ref{eq 16}) becomes
\begin{align}\label{eq 17}
v_D(z) p_D^2(z)= v_{A\circ T (D)}(0).
\end{align}
We now shift our focus into computing the bounds for $v_{A\circ T(D)}(0)$, for $v=c$ or $k$. We have already seen that $E_n\subset T(D)$ and $(1/c_n)\mbb B^n \subset A(E_n)$, which together imply $(1/c_n)\mbb B^n \subset A\circ T(D)$. Therefore the monotonicity property of the volume elements gives us
\begin{align}\label{eq 2}
v_{A\circ T(D)}(0)\leq v_{(1/c_n)\mbb B^n}(0).
\end{align} 
The map $z\mapsto c_nz$ is a biholomorphism between $(1/c_n)\mbb B^n$ and $\mbb B^n$ fixing the origin, and hence by the transformation rule (the equality in (\ref{ineq})) one obtains
\begin{align}\label{eq 3}
v_{(1/c_n)\mbb B^n}(0)=c_n^{2n} v_{\mbb B^n}(0)=c_n^{2n}.
\end{align}
Using Eq.~(\ref{eq 3}) and putting the value of $c_n$ in (\ref{eq 2}), we arrive at
\begin{align}\label{eq 4}
v_{A\circ T(D)}(0)\leq \left(\dfrac{4^n-1}{3}\right)^n \quad \text{for}\;\, v= c \;\text{or}\; k.
\end{align}
Now, the following map
\[
\Psi(z)=\left(\dfrac{z_1}{2-z_1},\ldots,\dfrac{z_n}{2-z_n}\right)
\]
can be checked to be a biholomorphism between $\left\{(z_1,\ldots,z_n)\in \mbb C^n: \Re z_j<1,\, j=1,\ldots,n\right\}$ and the unit polydisc $\mbb D^n$. It follows from Proposition~\ref{linear maps} that
\[
A\circ T(D)\subset \bigcap\limits_{j=1}^n \big\{z\in \mbb C^n: \Re z_j<1\big\},
\]
and hence we obtain
\begin{align*}
\Psi\circ A\circ T(D)\subset \Psi\left(\bigcap\limits_{j=1}^n \big\{z\in \mbb C^n: \Re z_j<1\big\}\right)=\mbb D^n.
\end{align*}
Therefore, using the transformation rule and monotonicity property of the volume elements,
\begin{align}\label{eq 5}
v_{A\circ T(D)}(0)=|\det \Psi'(0)|^2\, v_{\Psi\circ A\circ T(D)}(0)\geq |\det \Psi'(0)|^2\, v_{\mbb D^n}(0).
\end{align}
Since a scaling of $1/\sqrt{n}$ maps $\mbb D^n$ into $\mbb B^n$, one obtains by (\ref{ineq})
\begin{align}\label{eq 12}
v_{\mbb D^n}(0)\geq \dfrac{1}{n^n} v_{\mbb B^n}(0)=\dfrac{1}{n^n}.
\end{align}
Again, one can check that $\Psi'(0)$ is a diagonal matrix with each diagonal entry 1/2, and thus,
\[
\left|\det \Psi'(0)\right|^2=\dfrac{1}{2^{2n}}.
\]
We put these values in (\ref{eq 5}) which gives
\begin{align}\label{eq 6}
v_{A\circ T(D)}(0)\geq \dfrac{1}{(4n)^n}\quad \text{for}\;\, v= c \;\text{or}\; k.
\end{align}
Finally, using inequalities (\ref{eq 4}) and (\ref{eq 6}) in Eq.~(\ref{eq 17}), we obtain our desired estimates.
\end{proof}

As we have already derived the proposed upper and lower bounds for the volume elements as in Theorem~\ref{ge convex}, the proof of Theorem~\ref{com convex} will become evident by using a result of Nikolov and Pflug on the geometric estimates for the Bergman kernel on convex domains.

\begin{proof}[Proof of Theorem~\ref{com convex}]
For a non-degenerate convex domain $D\subset \mbb C^n, n\geq 2$, recall that in the previous proof we obtained
\begin{align*}
\dfrac{1}{(4n)^n} \leq v_D(z) p_D^2(z) \leq \left(\dfrac{4^n-1}{3}\right)^n.
\end{align*}
In \cite{Nik-Pflug-2003}, Nikolov and Pflug established the following inequalities to bound the Bergman kernel on any non-degenerate convex domain $D\subset \mbb C^n$:
\begin{align*}
\dfrac{1}{(4\pi)^n} \leq K_D(z) p_D^2(z) \leq \dfrac{(2n)!}{(2\pi)^n}.
\end{align*}
The estimates in Theorem~\ref{com convex} now become apparent by considering the previous two sets of inequalities and taking their ratios appropriately.
\end{proof}

\section{Bounds on $\mbb C$-convex domains}

In this section we prove Theorem~\ref{ge C-convex} and Theorem~\ref{com C-convex}.

\begin{proof}[Proof of Theorem~\ref{q C-convex}]
Here $D$ is a non-degenerate $\mbb C$-convex domain in $\mbb C^n$ and $z\in D$. As before, we assume $z=0$ by simply translating the domain. Since the map $A\circ T$ is a biholomorphism, we have
\begin{align}\label{eq 18}
v_D(z)=\left|\det (A\circ T)'(0)\right|^2 v_{A\circ T (D)}(0).
\end{align}
In the proof of Theorem~\ref{ge convex} we have seen that 
\begin{align*}
\left|\det (A\circ T)'(0)\right|^2= \dfrac{1}{\|p^1\|^2\cdots \|p^n\|^2},
\end{align*}
and thus Eq.~(\ref{eq 18}) becomes
\begin{align}\label{eq 22}
v_D(z) p_D^2(z)= v_{A\circ T (D)}(0).
\end{align}
Therefore we need to study the bounds for $v_{A\circ T(D)}(0)$, for $v=c$ or $k$. By the construction of $T$, the discs $\mbb D e_j\subset T(D)$ for $j=1,\ldots,n$. The $\mbb C$-convexity of $D$ implies $D$ is linearly convex, and therefore Lemma~15 of \cite{Nik-Pflug-Zwonek2} gives
\[
E_n\subset T(D).
\]
Again, as indicated in Sect.~\ref{prep}, we have $(1/c_n)\mbb B^n \subset A(E_n)$, and therefore $(1/c_n)\mbb B^n \subset A\circ T(D)$. Thus, owing to the monotonicity property of the volume elements,
\begin{align}\label{eq 8}
v_{A\circ T(D)}(0)\leq v_{(1/c_n)\mbb B^n}(0).
\end{align} 
In (\ref{eq 3}) we have already calculated 
\[
v_{(1/c_n)\mbb B^n}(0)=c_n^{2n}.
\]
Putting this in Eq.~(\ref{eq 8}) along with the value of $c_n$, we obtain
\begin{align}\label{eq 9}
v_{A\circ T(D)}(0)\leq \left(\dfrac{4^n-1}{3}\right)^n \quad \text{for}\;\, v= c \;\text{or}\; k.
\end{align}
Now, let us denote $\pi_j$ to be the projection of $\mbb C^n$ onto the $j$-th coordinate. Since complex affine maps preserve $\mbb C$-convexity, $A\circ T(D)$ is also a $\mbb C$-convex domain. Therefore, by Theorem~2.3.6 of \cite{APS}, $\Om_j:=\pi_j \left(A \circ T(D)\right)$ is a simply connected domain in the $z_j$-plane. Clearly
\begin{align*}
A\circ T(D)\subset \prod_{j=1}^n \Om_j.
\end{align*}
We now consider Riemann maps $\phi_j$ those transform $(\Om_j,0)$ onto $(\mbb D, 0)$ for $j=1,\ldots,n$. Note that domains $\Om_j$ satisfy $1\in \pa \Om_j$, and hence $\dist(0, \pa \Om_j)\leq 1$ for each $j=1,\ldots,n$. The Koebe 1/4 theorem then implies that
\begin{align}\label{eq 13}
\left|\phi_j'(0)\right|\geq \dfrac{1}{4} \quad \text{for}\;\, j=1,\ldots,n.
\end{align}
Next, consider the following map
\begin{align*}
\Phi(z_1,\ldots,z_n)=\big(\phi_1(z_1),\ldots, \phi_n(z_n)\big).
\end{align*}
It is obvious that $\Phi$ is a biholomorphism between $\prod_{j=1}^n \Om_j$ and $\mbb D^n$. We have already seen that $A\circ T(D)\subset \prod_{j=1}^n \Om_j$, and consequently we obtain
\begin{align*}
\Phi\circ A\circ T(D)\subset \Phi\left(\prod_{j=1}^n \Om_j\right)=\mbb D^n.
\end{align*}
Therefore, using the transformation rule and monotonicity property of the volume elements,
\begin{align}\label{eq 11}
v_{A\circ T(D)}(0)=\left|\det \Phi'(0)\right|^2\, v_{\Phi\circ A\circ T(D)}(0)\geq \left|\det \Phi'(0)\right|^2\, v_{\mbb D^n}(0).
\end{align}
We have already established $v_{\mbb D^n}(0)\geq 1/n^n$ (see (\ref{eq 12})). Again, note that
\begin{align*}
\left|\det \Phi'(0)\right|^2=\left|\phi_1'(0)\right|^2\cdots \left|\phi_n'(0)\right|^2,
\end{align*}
and by virtue of (\ref{eq 13}) we have
\begin{align*}
\left|\det \Phi'(0)\right|^2 \geq \dfrac{1}{4^{2n}}.
\end{align*}
Now put these values in (\ref{eq 11}) which gives
\begin{align}\label{eq 14}
v_{A\circ T(D)}(0)\geq \dfrac{1}{(16n)^n}\quad \text{for}\;\, v= c \;\text{or}\; k.
\end{align}
Finally, using inequalities (\ref{eq 9}) and (\ref{eq 14}) in Eq.~(\ref{eq 22}), we obtain our desired estimates.
\end{proof}

Similar to the proof of Theorem~\ref{com convex}, here we will use a result of Nikolov, Pflug and Zwonek, in the proof of Theorem~\ref{com C-convex}, which deals with the geometric estimates for the Bergman kernel on $\mbb C$-convex domains.

\begin{proof}[Proof of Theorem~\ref{com C-convex}]
By virtue of Theorem~\ref{ge C-convex} we have
\begin{align*}
\dfrac{1}{(16n)^n} \leq v_D(z) p_D^2(z) \leq \left(\dfrac{4^n-1}{3}\right)^n,
\end{align*}
for a non-degenerate $\mbb C$-convex domain $D\subset \mbb C^n$, for $n\geq 2$. Moreover, Nikolov \textit{et al.} in \cite{Nik-Pflug-Zwonek2} established the following inequalities to bound the Bergman kernel on any non-degenerate $\mbb C$-convex domain $D\subset \mbb C^n$:
\begin{align*}
\dfrac{1}{(16\pi)^n} \leq K_D(z) p_D^2(z) \leq \dfrac{(2n)!}{(2\pi)^n}.
\end{align*}
The estimates in Theorem~\ref{com C-convex} can now be easily deduced by considering the previous two sets of inequalities and taking their ratios appropriately.
\end{proof}

\noindent \textbf{Funding:} The author was supported by the National Board for Higher Mathematics (NBHM) Post-doctoral Fellowship (File no. 0204/10(7)/2023/R{\&}D-II/2777) and the Post-doctoral program at Indian Institute of Science, Bangalore.\\
 
\noindent \textbf{Data availability:} Data sharing is not applicable to this article as no datasets were generated or analysed during the current study.\\

\noindent \textbf{Conflict of interest:} The author declares that he has no conflict of interest.

\end{document}